\documentclass[final,reqno,10pt]{amsart}
\usepackage{amsmath,amsthm,amssymb}
\usepackage{color,soul}
\usepackage{tikz}
\usepackage{mathrsfs}

\newcommand{\R}{\mathbb{R}}
\newcommand{\C}{\mathbb{C}}
\newcommand{\N}{\mathbb{N}}
\newcommand{\Z}{\mathbb{Z}}
\newcommand{\SH}{S_\mathcal{H}}
\newcommand{\BH}{\mathcal{B}(\mathcal{H})}

\newcommand{\tr}{\operatorname{tr}}
\newcommand{\re}{\operatorname{Re}}
\newcommand{\im}{\operatorname{Im}}
\newcommand{\conv}{\operatorname{conv}}

\newcommand{\inner}[1]{\left\langle #1 \right\rangle}

\newcommand{\ext}{\operatorname{ext}}

\newcommand{\cl}[1]{\overline{#1}}
\newcommand{\spn}{\operatorname{span}}

\theoremstyle{plain}
\newtheorem{theorem}{Theorem}[section]

\newtheorem{corollary}[theorem]{Corollary}
\newtheorem{lemma}[theorem]{Lemma}
\newtheorem{proposition}[theorem]{Proposition}

\numberwithin{equation}{section}

\theoremstyle{definition}

\newtheorem{example}[theorem]{Example}

\theoremstyle{remark}
\newtheorem{remark}[theorem]{Remark}

\begin{document}
\title[Inverse continuity for Hilbert space operators]{Inverse continuity of the numerical range map for Hilbert space operators}
\author[B. Lins, I. M. Spitkovsky]{Brian Lins$^*$, Ilya M. Spitkovsky$^\dagger$}
\date{}
\address{Brian Lins, Hampden-Sydney College}
\email{blins@hsc.edu}
\address{Ilya M. Spitkovsky, New York University, Abu Dhabi}
\email{ims2@nyu.edu, imspitkovsky@gmail.com}

\thanks{$^*$Corresponding author.}
\thanks{$^\dagger$Supported in part by the Faculty Research funding from the Division of Science and Mathematics, New York University Abu Dhabi.}

\begin{abstract}
We describe continuity properties of the multivalued inverse of the numerical range map $f_A:x \mapsto \inner{ Ax, x }$ associated with a linear operator $A$ defined on a complex Hilbert space $\mathcal{H}$.  We prove in particular that $f_A^{-1}$ is strongly continuous at all points of the interior of the numerical range $W(A)$. We give examples where strong and weak continuity fail on the boundary and address special cases such as normal and compact operators.   
\end{abstract}

\maketitle

\section{Introduction}

Let $\mathcal{H}$ be a complex Hilbert space with inner product $\inner{ \cdot, \cdot }$ and norm $\|\cdot\| := \sqrt{\inner{x,x}}$.  Let $\BH$ denote the set of all bounded linear operators from $\mathcal{H}$ into $\mathcal{H}$ and let $\SH := \{x \in \mathcal{H} : \|x\| = 1 \}$ denote the unit sphere in $\mathcal{H}$.  For any operator $A \in \BH$, the \textit{numerical range map} of $A$ is the map $f_A:\SH \rightarrow \C$ such that $f_A(x) := \inner{ Ax,x }$. The \emph{numerical range} of $A$, denoted $W(A)$, is the image of $\SH$ under $f_A$, $W(A) := \{ f_A(x) : x \in \SH \}$. Throughout the paper, we use $\cl{S}$, $\partial S$, $\conv S$, and $\ext S$ to represent the closure, boundary, convex hull, and extreme points of a set $S$, respectively.  


In \cite{CJKLS2}, two notions of continuity were defined for the set-valued inverse numerical range map $f_A^{-1}$. We say that $f_A^{-1}$ is \emph{strongly continuous} at $z \in W(A)$ when the direct mapping $f_A$ is open in the relative topology of $W(A)$ at all pre-images $x \in f_A^{-1}(z)$. If there is at least one pre-image $x \in f_A^{-1}(z)$ for which $f_A$ is open, then $f_A^{-1}$ is \emph{weakly continuous} at $z$. Strong continuity is sometimes just called continuity in the literature on multivalued functions. The definition of weak continuity can be traced back to \cite{BraHe94}.  

S. Weis observed \cite{Weis16} that the continuity of certain maximal entropy inference maps on a quantum state space is equivalent to strong inverse continuity of a related numerical range map.  In this paper, we aim to extend the results of \cite{CJKLS2} and \cite{LLS13} to the infinite dimensional setting.  In the following section, we recall some important facts about numerical ranges and perturbation theory of operators.  In section 3, we prove that the inverse numerical range map is strongly continuous on $W(A)$ except, possibly, at certain extreme points on the boundary. We also give necessary and sufficient conditions for strong and weak continuity to hold on the numerical range of a normal operator. In section 4, we characterize strong and weak continuity for other points on the boundary of $W(A)$ under certain additional assumptions.  We conclude in section 5 with several examples.

\color{black}

\section{Preliminaries}

A linear operator $A$ defined on $\mathcal{H}$ has real and imaginary parts: $\re A := \frac{1}{2}(A+A^*)$ and $\im A := \frac{1}{2i}(A-A^*)$. For any $0 \le \theta < 2\pi$, the operator $\re (e^{-i \theta} A) = \cos \theta \re A + \sin \theta \im A$ is self-adjoint. The following proposition collects some well known results from perturbation theory about analytic self-adjoint operator-valued functions.  See, e.g., \cite[Chapter VII, section 3]{Kato}.  

\begin{proposition} \label{prop:Kato}
Let $A \in \mathcal{B}(H)$.  The operator valued function $\re(e^{-i\theta} A)$ is an analytic function of $\theta \in \R$, and its values are self-adjoint operators.  For any fixed $\theta_0$, each isolated eigenvalue $\lambda$ of $\re(e^{-i\theta_0} A)$ with finite multiplicity splits into one or several analytic functions $\lambda(\theta)$ that correspond to eigenvalues of $\re(e^{-i \theta} A)$ on an interval around $\theta_0$.  The corresponding spectral projections are also analytic functions of $\theta$.  
\end{proposition}

The spectrum $\sigma(T)$ of a self-adjoint operator $T \in \BH$ can be divided into the \emph{discrete spectrum} which consists of the isolated eigenvalues in $\sigma(T)$ with finite multiplicity and the \emph{essential spectrum} which is everything else in $\sigma(T)$ \cite[Section VII.3]{RS80}.  The analytic eigenvalue functions $\lambda(\theta)$ in Proposition \ref{prop:Kato} take values in the discrete spectrum of $\re(e^{-i\theta}A)$. Both the eigenvalue functions $\lambda(\theta)$ and the corresponding spectral projections $P(\theta)$ can be extended analytically to an interval in $\R$ as long as $\lambda(\theta)$ does not intersect the essential spectrum of $\re(e^{-i \theta} A)$ for any $\theta$ in that interval \cite[Cht. VII, section 3.2]{Kato}.  When $A$ is compact, this means that an eigenvalue function $\lambda(\theta)$ can be extended analytically as long as $\lambda(\theta) \ne 0$.  When we refer to an analytic eigenvalue function $\lambda(\theta)$ of $\re(e^{-i \theta} A)$ we will assume implicitly that $\lambda(\theta)$ is always part of the discrete spectrum on its domain.

For each analytic eigenvalue function $\lambda(\theta)$ with corresponding spectral projection $P(\theta)$, we can select an analytic path $\varphi(\theta)$ in $\mathcal{H}$ such that $P(\theta) \varphi(\theta) \ne 0$ on the interval where $P(\theta)$ is analytic. By scaling, we can construct a real analytic family of unit eigenvectors
\begin{equation} \label{eq:evectors}
x(\theta) := \frac{P(\theta) \varphi(\theta)}{\| P(\theta) \varphi(\theta) \|}
\end{equation} 
corresponding to the eigenvalues $\lambda(\theta)$.  The composition $f_A(x(\theta))$ parametrizes a real analytic curve that is contained in the numerical range of $A$. Following \cite{JAG98}, we will refer to such curves as  the \emph{critical curves of $A$}.  Below we derive a well-known expression for the critical curves in terms of the analytic eigenvalue functions of $\re(e^{-i \theta} A)$. 
\begin{align*}
f_A(x(\theta)) &= \inner{ Ax,x } \\
&= e^{i\theta} \inner{ e^{-i\theta} A x,  x} \\
&= e^{i\theta} \left(\inner{ \re(e^{-i\theta} A) x, x} + i \inner{ \im(e^{-i\theta} A) x, x}  \right). 
\end{align*}
Note that $ \frac{d}{d\theta} \re(e^{-i \theta} A) = \im(e^{-i \theta} A)$. Also $\inner{x(\theta), x'(\theta)} = 0$ for all $\theta$, since $\|x(\theta)\|= 1$ identically. Therefore,
\begin{align*}
\lambda'(\theta) &= \tfrac{d}{d\theta} \inner{ \re(e^{-i \theta} A) x(\theta), x(\theta) } \\
&= \inner{ \re(e^{-i \theta} A) x, x'} + \inner{ \im(e^{-i \theta} A) x,x} +\inner{\re(e^{-i \theta} A) x',x} \\
&= 2 \lambda(\theta) \inner{ x, x'} + \inner{ \im(e^{-i \theta} A) x,x } \\
&= \inner{ \im(e^{-i \theta} A) x,x} 
\end{align*} 
Combining these equations, we have
\begin{equation} \label{eq:critcurve}
f_A(x(\theta)) = e^{i\theta} \left( \lambda(\theta) + i\lambda'(\theta) \right). 
\end{equation}
If $\lambda(\theta)$ is the maximum (or minimum) eigenvalue of $\re(e^{-i \theta}A)$ for some $\theta$, then the corresponding point on the associated critical curve will be on the boundary of $W(A)$.  The following lemma gives a useful description of the boundary of $W(A)$.  These results are well known, so we won't prove them.  See \cite{Nar} for details.  
\begin{lemma} \label{lem:supportlines}
Let $A \in \BH$ and let $\mu(\theta)$ denote the maximum of the spectrum of $\re(e^{-i \theta}A)$ for all $\theta \in \R$. Then, 
$$\cl{W(A)} = \bigcap_{0 \le \theta < 2\pi} \{z \in \C : \re(e^{-i \theta} z) \le \mu(\theta)\}.$$
Any $z \in \partial W(A)$ lies on a tangent line 
\begin{equation} \label{eq:supportline}
L_\theta := \{z \in \C : \re(e^{-i \theta} z) = \mu(\theta)\}
\end{equation}
for some $\theta \in \R$. For $x \in \SH$, $f_A(x) \in L_\theta$ if and only if $x$ is an eigenvector of $\re(e^{-i \theta}A)$ corresponding to the eigenvalue $\mu(\theta)$. 
\end{lemma}

%
%

\section{General results}

We begin this section with a geometric lemma about spherical caps which is essentially the same as \cite[Lemma 3]{CJKLS2}. 
\begin{lemma} \label{lem:sphericalCap}
Let $S$ be the surface of a sphere in the Euclidean space $(\R^3, \| \cdot \|)$ and let $x
\in S$.  For any $\epsilon >0$, let $C := \{y \in S: \|x-y\| <
\epsilon \}$.  If $T: \R^3 \rightarrow \R^2$ is a linear
transformation, then $T(C)$ is convex and there is a $\delta > 0$ such that $\delta T(S) + (1-\delta) T(x) \subset T(C)$.
\end{lemma}
\begin{proof}
Observe that $C$ is the intersection of $S$ with an open half-space $H$. Therefore $\conv C = H \cap \conv S$.  Suppose $y \in \conv C$.  Choose $v \ne 0$ in the null space of $T$, and consider the line $y + tv, t\in \R$. 
At least one of the points where this line intersects $S$ will be in $H$.  Therefore there is a point $z \in C$ such that $T(z) = T(y)$ and so $T(C) = T(\conv C)$ which proves that $T(C)$ is convex.  As the center of the spherical cap $C$, $x$ is in the open set $H$. We may chose $\delta > 0$ small enough so that $\delta S + (1-\delta) x \subset \conv C$, so $\delta T(S) + (1-\delta)T(x) \subset T(C)$.  
\end{proof}

The following proposition generalizes \cite[Theorem 2]{CJKLS2} and \cite[Lemma 2.3]{LLS131} from the finite dimensional setting. Part of this result can be thought of as a generalization of the Toeplitz-Hausdorff Theorem.  Where the Toeplitz-Hausdorff Theorem guarantees that the image of $\SH$ under the numerical range map is a convex set, the proposition below says that neighborhoods in $\SH$ also have convex images. 
\begin{proposition} \label{prop:scaling}
Suppose that $A \in \BH$ and $z = f_A(x)$ where $x \in \SH$. Fix $\epsilon > 0$, and let $U = \{y \in \SH : \|y -x \| < \epsilon \}$ be the $\epsilon$-neighborhood around $x$ in $\SH$. Then $f_A(U)$ is convex, and there is a constant $\delta > 0$ such that $\delta W(A) + (1-\delta) z \subseteq f_A(U)$.   
\end{proposition}

\begin{proof}

Let $V$ be any two dimensional complex subspace of $\mathcal{H}$. By choosing an orthonormal basis for $V$, we may identify $V$ with $\C^2$, and $\mathcal{B}(V)$ with the set of 2-by-2 matrices $M_2(\C)$. Recall that $M_2(\C)$ has an inner product $\inner{X,Y} = \tr(Y^*X)$ and corresponding norm $\|X\| = \sqrt{\tr(X^*X)}$. The following equation holds for any $v, w \in V$ with $\|v\|=\|w\|=1$.
\begin{equation} \label{eq:projDist}
\begin{split}
\|vv^* - ww^*\|^2 &= \tr((vv^*-ww^*)(vv^*-ww^*)) \\
&= 2- 2 \tr(vv^*ww^*)\\
&= 2 - 2 \left| \inner{v,w} \right|^2.
\end{split}
\end{equation}

Let $\Sigma = \{ vv^* : v \in V, \|v \|=1 \}$. The set of self-adjoint operators in $\mathcal{B}(V)$ is a real vector space of dimension four and $\Sigma$ is the surface of a sphere with radius $\frac{1}{2}$ in the three dimensional affine subspace consisting of matrices with trace one \cite{Dav71}. 

Let $C = \{vv^* : v \in U \cap V \}$.  We will show that $C$ is a spherical cap in $\Sigma$. Observe that for $v \in V$ with $\|v \| = 1$,
\begin{equation} \label{eq:newdist}
\begin{split}
2\|v-x\|^2-\tfrac{1}{2}\|v-x\|^4 &= 2\left(2 - 2 \re \inner{x,v}\right) - \tfrac{1}{2}\left( 2 - 2 \re \inner{x,v} \right)^2 \\
&= 2 - 2 (\re \inner{x,v})^2 \\
&\ge 2 - 2 \left| \inner{x,v} \right|^2 = \|vv^*-xx^*\|^2.
\end{split}
\end{equation}  
Furthermore, equality holds in \eqref{eq:newdist} if and only if $\im \inner{x,v} = 0$.  Therefore $C$ is a subset of the set $\{vv^* \in \Sigma : \|vv^*-xx^*\|^2 < 2 \epsilon^2 - \tfrac{1}{2} \epsilon^4 \}$. At the same time, if $v \in V$ with $\|v\|=1$ satisfies $\|vv^*-xx^*\|^2 < 2 \epsilon^2 - \tfrac{1}{2} \epsilon^4$, then we may assume without changing the value of $vv^*$ that $\im \inner{x,v} = 0$.  This implies by \eqref{eq:newdist} that $\|v - x\| < \epsilon$ and therefore $vv^* \in C$. So $C = \{ vv^* \in \Sigma : \| vv^* - xx^* \| < 2 \epsilon^2 - \tfrac{1}{2} \epsilon^4 \}$ which is a spherical cap. This is true, even if $x \notin V$.   

Let $A_2 \in M_2(\C)$ be the compression of $A$ onto $V$. For any $v \in V$, we have 
\begin{equation} \label{eq:hatf}
f_A(v) = \inner{Av, v} = v^*A_2v = \tr(A_2vv^*).
\end{equation}
If we identify $\C$ with $\R^2$, then Lemma \ref{lem:sphericalCap} implies that the image of $C$ under the real linear transformation $X \mapsto \tr(A_2 X)$ is a convex set. Furthermore, if $x \in V$, then there is a $\delta > 0$ such that the image of $C$ also contains the image of $\delta \Sigma + (1-\delta) xx^*$. The constant $\delta$ can be selected based solely on the constant $\epsilon$ without regard to the particular subspace $V$.  Since $V$ may be chosen to contain any $y \in \SH$, we may apply \eqref{eq:hatf} to conclude that $f_A(U)$ contains $\delta W(A) + (1-\delta) z$.  Now consider a linearly independent pair $u,v \in U$. Let $V = \spn \{u, v \}$.  The image of the spherical cap $C$ corresponding to this subspace under the map $X \mapsto \tr(A_2 X)$ will be convex and will therefore contain the line segment connecting $f_A(u)$ to $f_A(v)$.  This proves that $f_A(U)$ is convex.  
\end{proof}

\begin{lemma} \label{lem:isolatedEP}
Let $A \in \BH$, $z \in W(A)$, and $0 < \delta < 1$. The set $\delta W(A) + (1-\delta) z$ contains a neighborhood of $z$ in the relative topology of $W(A)$ if and only if $z$ is not a limit point of $\ext \cl{W(A)}$.  
\end{lemma}
\begin{proof}
We will prove the equivalent statement: $z$ is a limit point of $\ext \cl{W(A)}$ if and only if $\delta W(A) + (1-\delta) z$ does not contain a neighborhood of $z$ in the relative topology of $W(A)$. To prove the forward implication, suppose that there is a sequence of extreme points $z_k \in \cl{W(A)} \backslash \{z \}$ that converges to $z$. As extreme points, no $z_k$ can be in $\delta \cl{W(A)} + (1-\delta) z$. Moreover $\delta \cl{W(A)} + (1-\delta) z$ is closed, so there is a neighborhood around each $z_k$ that contains an element $w_k$ in $W(A)$ but outside $\delta \cl{W(A)} + (1-\delta) z$. We can choose these $w_k$ so that they converge to $z$, proving that $\delta W(A) + (1-\delta) z$ does not contain a neighborhood of $z$ in $W(A)$.  

To prove the converse, suppose that $w_k \in W(A)$ is a sequence that converges to $z$ and that each $w_k$ is outside the set $\delta W(A) + (1-\delta)z$.  The ray from $z$ through $w_k$ intersects $W(A)$ in a line segment with one endpoint at $z$, and the other endpoint in $\partial W(A)$.  Let $z_k$ denote this endpoint.  Then $w_k = z + \lambda (z_k - z)$ where $0 \le \lambda \le 1$.  Since $w_k$ is not in $\delta W(A) + (1-\delta)z$, $\lambda$ must be at least $\delta$.  This means that $\delta |z_k-z| \le |w_k-z|$.  Since $w_k$ converges to $z$, so does $z_k$. Therefore $z \in \partial W(A)$. Each $z_k$ is either an extreme point of $\cl{W(A)}$, or it is a convex combination of two such extreme points, one of which must lie on the arc of $\partial W(A)$ between $z$ and $z_k$. This proves that $z$ is a limit point of $\ext \cl{W(A)}$.  
\end{proof}

The main result of this section now follows from Proposition \ref{prop:scaling}.

\begin{theorem}\label{thm:main}
Let $A \in \BH$ and let $z \in W(A)$. If $z$ is not a limit point of $\ext \cl{W(A)}$, then $f_A^{-1}$ is strongly continuous at $z$. 
\end{theorem}

\begin{proof}
Fix $x \in f_A^{-1}(z)$. By Proposition \ref{prop:scaling}, the image $f_A(U)$ of any neighborhood $U$ around $x$ in $\SH$ will contain $\delta W(A) + (1-\delta)z$ for some $\delta > 0$. Then Lemma \ref{lem:isolatedEP} implies that $f_A(U)$ contains a neighborhood of $z$ in the relative topology of $W(A)$, so $f_A^{-1}$ is strongly continuous at $z$.  
\end{proof}

In the finite dimensional setting, the numerical range of a normal matrix is a convex polygon, so Theorem \ref{thm:main} implies the inverse numerical range map is strongly continuous everywhere on that polygon. For normal operators defined on an infinite dimensional space, however, it is possible for strong and weak continuity of the inverse numerical range map to fail. The sufficient condition for strong continuity in Theorem \ref{thm:main} turns out to be necessary for weak continuity when $A$ is a normal operator.  

\begin{theorem} \label{thm:normal}
Let $A \in \BH$ be normal and $z \in W(A)$. If $z$ is a limit point of $\ext \cl{W(A)}$, then $f_A^{-1}$ is not weakly continuous at $z$. 
\end{theorem}
\begin{proof}
Since $W(A)$ is a convex set in the two dimensional real vector space $\C$, the set of extreme points of $\cl{W(A)}$ is closed. This means that $z \in W(A) \cap \ext \cl{W(A)}$. Any such $z$ is an eigenvalue of $A$ and any $x \in f_A^{-1}(z)$ is an eigenvector corresponding to $z$ \cite{Ber67}.  Fix one particular $x \in f_A^{-1}(z)$. Because the inner product is continuous, the set of $v \in \SH$ such that $|\inner{x,v}| > \frac{\sqrt{2}}{2}$ is an open neighborhood of $x$.  Any $v$ in this neighborhood can be decomposed as $v = \alpha x + \beta y$
where $\alpha = \inner{x,v}$, $y = \frac{v -\alpha x}{\|v-\alpha x\|}$ is orthogonal to $x$, and $\beta = \|v-\alpha x\|$. The operator $A$ is normal, so $x$ is also an eigenvector of $A^*$ \cite[Theorem 12.12]{RudinFA}. This lets us calculate $f_A(v)$: 
\begin{align*}
\inner{Av,v} &= \inner{\alpha Ax + \beta Ay, \alpha x + \beta  y}  \\
&= |\alpha|^2 \inner{Ax, x} + \alpha \overline{\beta} \inner{Ax, y} + \beta \overline{\alpha} \inner{Ay,x} + |\beta|^2 \inner{Ay, y}  \\
&= |\alpha|^2 \inner{Ax, x} + \alpha \overline{\beta} \inner{Ax,y} + \beta \overline{\alpha} \inner{y,A^*x} + |\beta|^2 \inner{Ay, y}.  \\
 &= |\alpha|^2 \inner{Ax,x} + |\beta|^2 \inner{Ay,y} \\
&= |\alpha|^2 z + |\beta|^2 \inner{Ay,y}.
\end{align*} 
Since $|\alpha|^2 + |\beta|^2 = 1$ and $|\alpha|^2 > \frac{1}{2}$, we conclude that $f_A(v) \in \frac{1}{2}W(A)+ \frac{1}{2}z$ for any $v$ in this neighborhood around $z$.  By Lemma \ref{lem:isolatedEP}, $\frac{1}{2}W(A)+ \frac{1}{2}z$ does not contain a neighborhood of $z$ in the relative topology of $W(A)$.  Therefore $f_A$ is not an open mapping at $x$ which means that $f_A^{-1}$ is not weakly continuous at $z$.  
\end{proof}

\section{Inverse Continuity on the Boundary}
In this section we investigate strong and weak inverse continuity for points on the boundary of the numerical range that are not covered by Theorem \ref{thm:main}. Let us begin with a review of what is known about the boundary of $W(A)$ when $A$ is a bounded operator on an infinite dimensional Hilbert space. Recall that the \emph{essential numerical range} of an operator $A \in \BH$ is the set $W_e(A) = \bigcap_{K \in \mathcal{K}(\mathcal{H})} \cl{W(A+K)}$ where $\mathcal{K}(\mathcal{H})$ is the set of compact operators on $\mathcal{H}$. It is readily apparent that $W_e(A)$ is a closed, convex subset of $\cl{W(A)}$.  Fillmore et al. observed \cite{FiStWi} that $z \in W_e(A)$ if and only if there is a sequence $x_n \in \SH$ such that $x_n$ converges weakly to $0$ while $f_A(x_n) \rightarrow z$. 

P. Lancaster \cite{Lancaster75} described the relationship between the boundary of the numerical range and the essential numerical range. By \cite[Theorem 1]{Lancaster75}, the extreme points of $\cl{W(A)}$ are contained in $W(A) \cup W_e(A)$. 


\begin{lemma} \label{lem:M}
Let $A \in \BH$ and let $M$ denote the set of angles $\theta$ for which the maximum value of the spectrum of $\re(e^{-i \theta} A)$ is an isolated eigenvalue with finite multiplicity. Let $L_\theta$ be defined as in \eqref{eq:supportline}. Then $\theta \in M$ if and only if  $L_\theta \cap \cl{W(A)}$ does not contain elements of $W_e(A)$.
\end{lemma}

\begin{proof}
For $\theta \in \R$, let $\mu(\theta)$ denote the maximum of the spectrum of $\re(e^{-i \theta}A)$. If $\mu(\theta)$ is not an isolated eigenvalue with finite multiplicity of $\re(e^{-i \theta} A)$, then $\mu(\theta)$ is in the essential spectrum of $\re(e^{-i \theta} A)$.  By Weyl's criterion \cite[Theorem VII.12]{RS80} there is a sequence $x_k \in \SH$ such that $\|\re(e^{-i \theta} A) x_k - \mu(\theta) x_k\| \rightarrow 0$ as $k\rightarrow \infty$ while $x_k$ converges weakly to 0. Then $\inner{\re(e^{-i \theta} A)x_k,  x_k} \rightarrow \mu(\theta)$.  By passing to a subsequence, we can assume that $\inner{Ax_k, x_k}$ also converges, and the limit will be an element of $W_e(A)$ that is contained in $L_\theta$.  

Conversely, suppose that $\mu(\theta)$ is an isolated eigenvalue of $\re(e^{-i \theta} A)$ with finite multiplicity.  There is a compact self-adjoint operator $K$ such that the maximum element of the spectrum of $\re(e^{-i \theta} A) + K$ is strictly less than $\mu(\theta)$.  Then $\cl{W(A+K})$ does not intersect $L_\theta$, so $L_\theta$ cannot contain an element of $W_e(A)$.  
\end{proof}

\color{black}

The following proposition is a summary of several results of Narcowich \cite{Nar}, restated in terms of the essential numerical range with the help of Lemma \ref{lem:M}. 

\begin{proposition} \label{prop:bdry}
Let $A \in \BH$. Any connected subset of $\partial W(A)$ that is separated from $W_e(A)$ is a piecewise analytic curve. Each analytic portion is either a line segment or can be parameterized by $\mu(\theta) + i \mu'(\theta)$ on an open interval where $\mu(\theta)$ is the maximum element of the spectrum of $\re(e^{-i \theta} A)$ and $\mu(\theta)$ is also an isolated eigenvalue with finite multiplicity on that interval. The points where the curve is not analytic may accumulate, but only at endpoints of line segments in $\partial W(A)$ that also contain elements of $W_e(A)$. In particular, if $W_e(A) \cap \partial W(A) = \varnothing$, then $\partial W(A)$ is a finite union of analytic curves. 
\end{proposition}

Each curved analytic portion of $\partial W(A)$ described above is a critical curve corresponding to the maximal eigenvalue of $\re(e^{-i \theta} A)$ on an interval of values of $\theta$. Flat analytic portions correspond to angles $\theta$ where the maximal eigenvalue function splits into two or more eigenvalue functions with different slopes.  We refer the interested reader to \cite{Nar} for more details. 

\begin{lemma} \label{lem:open}
Let $A \in \BH$. For each $z \in \ext W(A)$, let $P(z)$ denote the orthogonal projection onto the closure of the span of $f_A^{-1}(z)$. If $f_A$ is an open mapping in the relative topology of $W(A)$ at $x \in \SH$ and $f_A(x) = z \in \ext W(A)$, then for any sequence $z_k \in \ext W(A)$ that converges to $z$, $P(z_k) x \rightarrow x$. 
\end{lemma}

\begin{proof}
Suppose by way of contradiction that $P(z_k) x$ does not converge to $x$. We may assume by passing to a subsequence that there is an $\epsilon > 0$ such that $\|P(z_k)x - x \| > 2\epsilon$ for all $z_k$. Choose any $z_k$ and $y \in f_A^{-1}(z_k)$.  Then
\begin{align*}
2\epsilon &< \| x - P(z_k) x \| & \\
&\le \|x - y \| + \|y - P(z_k) x\| & \text{(Triangle inequality)}\\
&= \|x - y \| + \|P(z_k) y - P(z_k) x\| & \text{(Since $P(z_k) y = y$)} \\
&\le \|x - y \| + \|P(z_k)\| \|x- y  \|  &  \\
&\le 2 \|x-y \|.
\end{align*}
In particular, $y$ is not in the neighborhood $U = \{y \in \SH : \|x-y\| < \epsilon \}$ around $x$. Therefore $f_A(U)$ does not contain any $z_k$, so $f_A$ is not an open mapping at $x$.    
\end{proof}

\begin{remark}
In the statement of Lemma \ref{lem:open}, we defined $P(v)$ to be the orthogonal projection onto the closure of the span of $f_A^{-1}(v)$.  In fact, the span of $f_A^{-1}(v)$ is always a closed subspace when $v \in \ext W(A)$, so it is redundant to refer to its closure. 
\end{remark}

The main result of this section follows. It extends \cite[Theorem 2.1]{LLS131} to operators in infinite dimensions. This theorem does not completely characterize when weak and strong continuity hold on $\partial W(A)$ because it only applies to points where the corresponding maximal eigenvalue of $\re(e^{-i \theta} A)$ is isolated and has finite multiplicity. 

\begin{theorem} \label{thm:compactchar}
Let $A \in \BH$ and $z \in W(A) \cap \ext \cl{W(A)}$. Let $L_\theta$ be defined as in Lemma \ref{lem:supportlines}. If $z \in L_{\theta_0}$ where the maximum of the spectrum of $\re(e^{-i \theta_0} A)$ is an isolated eigenvalue with finite multiplicity, then 
\begin{enumerate}
\item $f_A^{-1}$ is strongly continuous at $z$ if and only if $z$ is contained in only one critical curve of $A$; 
\item $f_A^{-1}$ is weakly continuous at $z$ if and only if $\partial W(A)$ is analytic at $z$ or $z$ is an endpoint of a flat portion of $\partial W(A)$.
\end{enumerate}
\end{theorem}

\begin{proof}
Since $z \in L_{\theta_0}$, $\re(e^{-i \theta_0} z)$ is the maximum eigenvalue of $\re(e^{-i\theta_0}A)$.  
If we perturb the angle $\theta$, the maximum eigenvalue of $\re(e^{-i \theta} A)$ at $\theta = \theta_0$ may split into one or more eigenvalue functions $\lambda(\theta)$ that are analytic in a neighborhood of $\theta_0$. Each of these eigenvalue functions corresponds to a critical curve given by $\lambda(\theta) + i \lambda'(\theta)$. If all of the eigenvalue functions $\lambda(\theta)$ have the same slope at $\theta_0$, then $z$ is the unique point where $L_{\theta_0}$ intersects $W(A)$.  If the eigenvalue functions have different slopes at $\theta_0$, then the intersection of $L_{\theta_0}$ with $W(A)$ will be a flat portion.  In that case, $z$ must be an endpoint of the flat portion, since we have assumed that $z \in \ext W(A)$.  
From here, we divide the proof into three cases.

\begin{enumerate}
\item[I.] $z$ is in the relative interior of a curved analytic arc of $\partial W(A)$.  
\item[II.] $z$ is a singularity where one curved analytic arc of $\partial W(A)$ transitions to a flat portion of the boundary.  
\item[III.] $z$ is a singularity where one curved analytic arc of $\partial W(A)$ transitions to another.  
\end{enumerate}

We will show that $f_A^{-1}$ is weakly continuous at $z$ in cases I and II, while weak continuity fails in case III. We begin with case I. Suppose that $z$ is contained in the relative interior of one of the analytic curves defining the boundary of $W(A)$. Any such curve will be a critical curve of $A$. This critical curve can be expressed as $f_A(x(\theta))$ for some analytic family of eigenvectors $x(\theta)$ of $\re(e^{-i \theta}A)$. Then $z = f_A(x_0)$ where $x_0 = x(\theta_0)$.  If we take a neighborhood $U$ around $x_0$ in $\SH$, and consider $f_A(U)$, then by Proposition \ref{prop:scaling}, $f_A(U)$ is a convex subset of $W(A)$ that contains a neighborhood of the boundary around $z$.  Therefore, it contains a neighborhood of $z$ in the relative topology of $W(A)$, proving that $f_A^{-1}$ is weakly continuous at $z$.  

Now consider case II, where $z$ is the transition from a flat portion to a curved analytic portion of the boundary.  The curved portion is one of the critical curves of $A$, and can be parameterized by $f_A(x(\theta))$ for some analytic family of eigenvectors of $\re(e^{-i \theta} A)$.  Again $z = f_A(x_0)$ where $x_0 = x(\theta_0)$.  Let $U$ be a neighborhood of $x_0$ in $\SH$.  The image $f_A(U)$ contains a neighborhood of $z$ on the curved portion of the boundary.  It is also a convex set that contains all points of $W(A)$ in a neighborhood of $z$ on the flat portion of the boundary by Proposition \ref{prop:scaling}.  We conclude that $f_A(U)$ contains a neighborhood of $z$ in the relative topology of $W(A)$ and therefore $f_A^{-1}$ is weakly continuous at $z$.  

It remains to prove that weak continuity fails in case III, that is, when $z$ is the transition point for two different curved analytic portions of the boundary of $W(A)$.  The two boundary portions adjacent to $z$ will be given by distinct critical curves of $A$.  Let $\lambda(\theta)$ and $\mu(\theta)$ denote the analytic eigenvalue functions of $\re(e^{-i \theta}A)$ corresponding to the two critical curves and let $P(\theta)$ and $Q(\theta)$ be their respective spectral projections. Choose any $x \in f_A^{-1}(z)$. Since $x$ cannot be in the range of both $P(\theta)$ and $Q(\theta)$, Lemma \ref{lem:open} implies that the map $f_A$ is not open at $x$. Therefore $f_A^{-1}$ is not weakly continuous at $z$.

We have completed the characterization of weak continuity, but we still need to verify the conditions for strong continuity to hold in cases I and II. Suppose there is only one critical curve that passes through $z$. Let $P(\theta)$ denote the spectral projection corresponding to this critical curve. If $x \in f_A^{-1}(z)$, then $x$ is in the range of $P(\theta_0)$. Choose an analytic path $\varphi(\theta)$ in $\SH$ such that $\varphi(\theta_0) = x$ and $P(\theta) \varphi(\theta) \ne 0$ for all $\theta$ in a neighborhood of $\theta_0$.  Then $x(\theta)$ given by \eqref{eq:evectors} is an analytic family of unit eigenvectors of $\re(e^{-i \theta}A)$, and the curve $f_A(x(\theta))$ is the unique critical curve passing through $z$.   

Choose a neighborhood $U = \{y \in \SH : \|y-x\| \le \epsilon\}$ around $x$. Since the eigenvectors $x(\theta)$ depend continuously on $\theta$, the image $f_A(U)$ contains a neighborhood of $z$ in the critical curve that passes through $z$. If $z$ happens to be the endpoint of a flat portion of $\partial W(A)$, then $f_A(U)$ also contains a neighborhood of $z$ in that flat portion by Proposition \ref{prop:scaling}. Therefore $f_A(U)$ contains a neighborhood of $z$ on the boundary of $W(A)$. Since $f_A(U)$ is convex by Proposition \ref{prop:scaling}, we conclude that $f_A(U)$ contains a neighborhood of $z$ in the relative topology of $W(A)$.  This proves that $f_A^{-1}$ is strongly continuous at $z$.  

Conversely, suppose that more than one critical curve contains $z$. In case I, one of these critical curves parameterizes $\partial W(A)$ in a neighborhood of $z$, while in case II, one of the critical curves parameterizes $\partial W(A)$ to one side of $z$, while the portion of the boundary on the other side of $z$ is flat. In any event, each of these critical curves corresponds to an eigenvalue function that is analytic in a neighborhood of $\theta_0$.  Let $\mu(\theta)$ denote the eigenvalue function corresponding to the critical curve that is part of the boundary near $z$. Let $\lambda(\theta)$ be the eigenvalue function corresponding to one of the other critical curves, and let $P(\theta)$ and $Q(\theta)$ denote the analytic families of spectral projections corresponding to $\lambda(\theta)$ and $\mu(\theta)$, respectively.  Since $P(\theta)$ and $Q(\theta)$ correspond to different spectral subspaces of the self-adjoint operator $\re(e^{-i \theta}A)$, their ranges are orthogonal. We may choose a pre-image $x \in f_A^{-1}(z)$ such that $x$ is in the range of $P(\theta_0)$, that is $P(\theta_0) x = x$. Then $Q(\theta) x \rightarrow 0$ as $\theta \rightarrow \theta_0$. By Lemma \ref{lem:open}, $f_A$ is not open at $x$, and therefore $f_A^{-1}$ is not strongly continuous at $z$.  
\end{proof}

The next result is an immediate consequence of Proposition \ref{prop:bdry} and Theorem \ref{thm:compactchar}.  

\begin{corollary} \label{thm:finite}
Let $A \in \BH$. If $W_e(A) \cap  \partial W(A) = \varnothing$, then there are at most finitely many points where strong (and thus weak) inverse continuity of $f_A^{-1}$ can fail. 
\end{corollary}

\begin{remark}
For any compact operator $A$ defined on an infinite dimensional Hilbert space, $W_e(A) = \{0\}$.  
If $0 \notin \partial W(A)$, then Proposition \ref{prop:bdry} implies that the boundary of $W(A)$ is a finite union of critical curves, and Theorem \ref{thm:compactchar} gives a complete description of when weak and strong continuity hold for the inverse numerical range map on the boundary.  
It is not clear what the necessary and sufficient conditions are for $f_A^{-1}$ to be strongly or weakly continuous at 0 when $0 \in \partial W(A) \cap W(A)$.  It is also not clear what happens at the opposite end point of a flat portion of the boundary of $W(A)$ that also contains 0.  These open questions are related to the possible structure of the boundary of the numerical range of a compact operator near the origin.  
\end{remark}

\section{Examples}

\subsection*{Normal operators}
\begin{example}
 Let $e_k$, $k \in \Z$, denote the standard orthonormal basis for $\mathcal{H}=\ell_2(\Z)$ and consider the compact normal operator $A:\ell_2 \rightarrow \ell_2$ defined by
$$A(e_k) := \begin{cases} 
0 & \text{ if } k = 0,\\
\frac{1}{k} + i \frac{1}{k^2} & \text{otherwise.}
\end{cases}$$
Since $0 \in \ext \cl{W(A)}$ and 0 is not an isolated element of $\ext \cl{W(A)}$, $f_A^{-1}$ is not weakly continuous at 0 by Theorem \ref{thm:normal}.  
\end{example}

\begin{example}
Let $T:\ell_2(\N) \rightarrow \ell_2(N)$ be the normal operator $(T(x))_k = \tau^k x_k$ where $\tau$ is an irrational root of unity.  Then $W(T)$ is the union of the open unit disk with the set $\{\tau^k : k \in \N \}$.  Each $\tau_k$ is an extreme point of $\cl{W(T)}$, and none of these extreme points is isolated.  By Theorem \ref{thm:normal}, $f_T^{-1}$ is not weakly continuous at any of these extreme points.  

\end{example}

\subsection*{Non-normal compact operators}
\begin{example}
The numerical range of the 4-by-4 matrix 
$$A = \begin{bmatrix}
0 & ik & 0 & 0 \\
ik & ib & 0 & 0 \\
0 & 0 & 0 & ik \\ 
0 & 0 & ik & -ib 
\end{bmatrix},$$
where $b,k > 0$, is the convex hull of two ellipses and weak continuity fails for $f_A^{-1}$ at 0 \cite[Example 9]{CJKLS2}.  We can choose $b$ and $k$ small enough so that the numerical range of $I_4 - A$ is contained in the unit circle. Here we use $I_4$ to denote the identity matrix on $\C^4$ while $I$ will denote the identity on the Hilbert space $\ell_2(\N)$. We then consider the operator 
$$T = -I + \bigoplus_{k=1}^\infty \left(I_4 - \tfrac{1}{k}A \right) e^{i\pi/k}.$$
Observe that $T$ is a compact operator on $ \ell_2(N)$.  Weak continuity of $f_T^{-1}$ fails at each of the points $e^{i \pi/k}-1$, $k \in \N$ by Theorem \ref{thm:compactchar}. So $T$ is an example of a compact operator with infinitely many weak continuity failures.  
\end{example}

\begin{example}
Let $\mathcal{H} = L^2(0,1)$.  The Volterra operator $V: \mathcal{H} \rightarrow \mathcal{H}$ is 
$$(Vf)(t) := \int_0^t f(s) \, ds.$$
It is well known that the Volterra operator is a compact linear operator. Halmos points out \cite[Problem 150]{Hal67} (see also \cite[Example 9.3.13]{Dav07}) that the numerical range of $V$ is the closed set $W(V)$ lying between the curves $$t \mapsto \frac{1-\cos(t)}{t^2} \pm i \frac{t- \sin t}{t^2}, ~~~~ 0 \le t \le 2\pi,$$
where the values at $t = 0, 2\pi$ are taken to be the corresponding limits. There is a flat portion on the imaginary axis with endpoints $i/(2\pi)$ to $-i/(2\pi)$. We will show that the inverse numerical range map is strongly continuous everywhere on $W(V)$.   

\begin{figure}[ht]
\begin{center}
\includegraphics[scale=0.2]{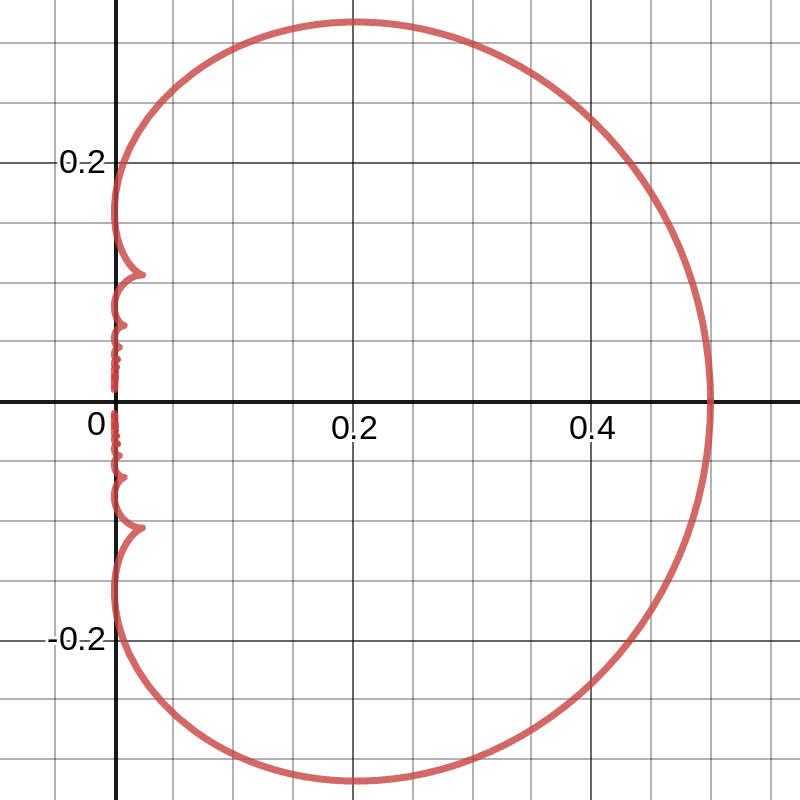}
\end{center}
\caption{The critical curves of the Volterra operator.}
\label{fig:Volterra}
\end{figure}

Let us review some facts about the Volterra operator, see \cite[Example 9.3.13]{Dav07} for details. The adjoint of $V$ is $(V^*f)(t) = \int_t^1 f(s) \, ds$. Let $V_\theta$ denote the real part of $e^{-i \theta} V$ and note that  
$$V_\theta = \tfrac{1}{2} ( e^{-i \theta} V + e^{i \theta} V^* ).$$
As long as $\theta \ne 0$ or $\pi$, the eigenvalues and corresponding unit eigenvectors of $V_\theta$ are 
$$\lambda_n = \frac{\sin \theta}{ 2\theta + 2n \pi}, ~ f_n(t) = e^{-i t (2 \theta + 2n \pi)} \text{ where } n \in \Z.$$
In particular, each eigenvalue has a one dimensional eigenspace and so the critical curves corresponding to each eigenvalue are well defined (See Figure \ref{fig:Volterra}).  The boundary of the numerical range is the critical curve corresponding to the maximal eigenvalue and therefore the inverse numerical range map is strongly continuous at all points on the boundary curve, except possibly the two endpoints by Theorem \ref{thm:compactchar}. Strong continuity also holds at all points in the relative interior of the flat portion of the boundary by Theorem \ref{thm:main}.  All that remains is to verify that $f_V^{-1}$ is strongly continuous at the endpoints of the flat portion, $\pm i/(2\pi)$.  

The real part of $V$ is the rank one orthogonal projection onto the constant function.  Let $\mathcal{H}_0 = \{f \in \mathcal{H} : \int_0^1 f(t) \, dt = 0 \}$.  The compression of $V$ onto $\mathcal{H}_0$ is a normal operator, and the functions $e^{-2\pi i n t}$, $n \in \Z \backslash \{0 \}$ are an orthonormal basis of eigenvectors with corresponding eigenvalues $i/(2\pi n)$.  

The boundary curve of $W(V)$ can be parameterized by $f_V(e^{-2it \theta})$ for $\theta \in [\pi,-\pi]$.  The functions $e^{-2it \theta}$ are the the unique pre-images of corresponding points on the boundary curve. This is true, even at the endpoints where $\theta = \pm \pi$.  Because the image of $f_V$ at $e^{\mp 2 \pi i t}$ contains a neighborhood of the boundary of $W(V)$ around $\pm i/(2\pi)$, it follows that $f_V$ is strongly continuous at both $\pm i/(2\pi)$.  

  
\end{example}

\subsection*{Weighted shift operators}
\begin{example}
The numerical range of a weighted shift operator is either an open or closed circular disk centered at the origin \cite[Proposition 16]{Shi74}. Conditions for determining whether the disk is closed or open can be found in \cite{St83}. We will demonstrate that the inverse numerical range map $f_A^{-1}$ of any weighted shift operator $A$ is strongly continuous everywhere in $W(A)$.   

Let $A$ be a weighted shift operator on $\ell_2(\Z)$ such that there is a bounded sequence of scalars $\alpha_k \in \C$ for which $(Ax)_{k+1} = \alpha_k x_k$.  If $W(A)$ is open, then $f_A^{-1}$ is strongly continuous on $W(A)$ by Theorem \ref{thm:main}. Suppose therefore that $W(A)$ is closed and choose $x \in \ell_2(\Z)$ with $\|x\|=1$ such that $\inner{ Ax,x }$ is equal to the numerical radius $\omega(A) := \sup \{|\inner{Ay,y}| : y \in \ell_2(\Z), \|y \|=1 \}$.  Fix $\tau \in \C$ with $|\tau| = 1$.  Let $y \in \ell_2$ be defined by $y_k = \tau^k x_k$.  Note that $f_A(y) = \inner{ Ay,y } = \sum_k \alpha_k \tau^{k+1} \bar{\tau}^{k} = \tau \inner{ Ax,x}$. Fix $\epsilon > 0$ and note that 
$$\|x-y\|^2 = \sum_{k \in \Z} |\tau^k -1|^2 |x_k|^2 \le \sum_{-N \le k \le N} |\tau^k-1|^2 + \tfrac{1}{2}\epsilon$$
for some $N$ sufficiently large. When $\tau$ is sufficiently close to 1, $|\tau^k-1|^2 \le \frac{\epsilon}{2(2N+1)}$ for all $k \in \{-N,\ldots,N\}$. In that case, $\|x-y\|^2 \le \epsilon.$
Therefore the map $f_A$ is relatively open at $x$ since the image of a neighborhood of $x$ contains a neighborhood of $\omega(A)$ on the boundary of $W(A)$.  It follows that $f_A^{-1}$ is strongly continuous at $\omega(A)$. By rotational symmetry, $f_A^{-1}$ is strongly continuous at all points of the boundary of $W(A)$ that are part of the numerical range.  It is worth mentioning that there are weighted shift operators where $W_e(A) = W(A)$ and for such operators the strong continuity of $f_A^{-1}$ on the boundary cannot be derived from Theorem \ref{thm:compactchar}. See \cite[Note V.4]{St83} for details on the construction of such examples.  
\end{example}

\bibliography{master}
\bibliographystyle{plain}

\end{document}